\documentclass[12pt,draft]{article}
\usepackage{a4wide}

\usepackage{amsfonts,amsmath}
\usepackage[dvips]{graphicx}

\def\COMMENT#1{}
\let\COMMENT=\footnote

\newcommand{\eps}{\varepsilon}
\newcommand{\s}{\alpha_t(G_{n,p})}

\newcommand{\atp}[1]{\alpha_{#1,p}(n)}
\newcommand{\sk}{\alpha_t^{(k)}(G_{n,p})}
\newcommand{\prob}{\,\mathbb{P}}
\newcommand{\Pa}{{\mathcal P}}
\newcommand{\ex}{\,\mathbb{E}}
\newcommand{\la}{\lambda}
\newcommand{\D}{{\mathcal D}}

\newtheorem{firsttheorem}{Proposition}

\newtheorem{theorem}[firsttheorem]{Theorem}
\newtheorem{lemma}[firsttheorem]{Lemma}

\begin{document}
\title{The $t$-stability number of a random graph}
\author{
Nikolaos Fountoulakis \\
\small Max-Planck-Institut f\"ur Informatik \\
\small Campus E1 4 \\ 
\small Saarbr\"ucken 66123 \\
\small Germany \\
\and Ross J. Kang\footnote{Part of this work was completed while this author was a doctoral student at the University of Oxford; part while he was a postdoctoral fellow at McGill University. He was supported by NSERC (Canada) and the Commonwealth Scholarship Commission (UK).} \\
\small School of Engineering and Computing Sciences \\
\small Durham University \\
\small South Road, Durham  DH1 3LE \\
\small United Kingdom \\
\and Colin McDiarmid \\
\small Department of Statistics \\
\small University of Oxford \\
\small 1 South Parks Road \\ 
\small Oxford OX1 3TG \\
\small United Kingdom
}
\date{26 October 2010\\
\small Mathematics Subject Classification: 05C80, 05A16} 
\maketitle
\begin{abstract}
Given a graph $G = (V,E)$, a vertex subset $S \subseteq V$ is called {\em $t$-stable} (or {\em $t$-dependent}) if the subgraph $G[S]$ induced on 
$S$ has maximum degree at most $t$.  The {\em $t$-stability number} $\alpha_t(G)$ of $G$ is the maximum order of a $t$-stable set in $G$. 
The theme of this paper is the typical values that this parameter takes on a random graph on $n$ vertices and edge probability equal to $p$.  
For any fixed $0 < p < 1$ and fixed non-negative integer $t$, we show that, with probability tending to $1$ as $n\to \infty$, the $t$-stability number takes on
 at most two values which we identify as functions of $t$, $p$ and $n$.  The main tool we use is an asymptotic expression for the expected number of $t$-stable sets of order $k$. We derive this expression by performing a precise count of the number of graphs on $k$ vertices that have maximum degree at 
most $k$.
\end{abstract}
\section{Introduction}
Given a graph $G = (V,E)$, a vertex subset $S \subseteq V$ is called {\em $t$-stable} (or {\em $t$-dependent}) if the subgraph $G[S]$ induced on $S$ has
 maximum degree at most $t$.  The {\em $t$-stability number} $\alpha_t(G)$ of $G$ is the maximum order of a $t$-stable set in $G$.  The main topic of
 this paper is to give a precise formula for the $t$-stability number of a dense random graph. 

The notion of a $t$-stable set is a generalisation of the notion of a stable set. Recall that a set of vertices $S$ of a graph $G$ is \emph{stable} if no two of its
 vertices are adjacent. In other words, the maximum degree of $G[S]$ is 0, and therefore a stable set is a 0-stable set.

The study of the order of the largest $t$-stable set is motivated by the study of the \emph{$t$-improper chromatic number} of a graph. A $t$-\emph
{improper colouring} of a graph $G$ is a vertex colouring with the property that every colour class is a $t$-stable set, and the $t$-\emph{improper chromatic
 number} $\chi_t(G)$ of $G$ is the least number of colours necessary for a $t$-improper colouring of $G$.  Obviously, a 0-improper colouring is a proper
 colouring of a graph, and the 0-improper chromatic number is the chromatic number of a graph.

The $t$-improper chromatic number is a parameter that was introduced and studied independently by Andrews and Jacobson~\cite{AnJa85}, Harary and
 Fraughnaugh (n\'ee Jones)~\cite{Har85,HaJo85}, and by Cowen \emph{et al.}~\cite{CCW86}.  The importance of the $t$-stability number in relation to the
 $t$-improper chromatic number comes from the following obvious inequality: if $G$ is a graph that has $n$ vertices, then 
\begin{align*}
\chi_t(G) \ge \frac{n}{\alpha_t(G)}.
\end{align*}

The $t$-improper chromatic number also arises in a specific type of radio-frequency assignment problem. Let us assume that the vertices of a given graph represent transmitters and an edge between two vertices indicates that the corresponding transmitters interfere. Each interference creates some amount of noise which we denote by $N$.  Overall, a transmitter can tolerate up to a specific amount of noise which we denote by $T$. The problem now is to assign frequencies to the transmitters and, more specifically, to assign as few frequencies as possible, so that we minimise the use of the electromagnetic spectrum. Therefore, any given transmitter cannot be assigned the same frequency as more than $T/N$ nearby transmitters --- that is, neighbours in the transmitter graph --- as otherwise the excessive interference would distort the transmission of the signal. In other words, the vertices/transmitters that are assigned a certain frequency must form a $T/N$-stable set, and the minimum number of frequencies we can assign is the $T/N$-improper chromatic number.    
  
Given a graph $G=(V,E)$, we let $S_{t}=S_{t}(G)$ be the collection of all subsets of $V$ that are $t$-stable.  We shall determine the order of the largest
 member of $S_t$ in a random graph $G_{n,p}$. Recall that $G_{n,p}$ is a random graph on a set of $n$ vertices, which we assume to be $V_n:=\{1,\ldots
 ,n\}$, and each pair of distinct vertices is present as an edge with probability $p$ independently of every other pair of vertices.  Our interest is in dense random
 graphs, which means that we take $0<p<1$ to be a fixed constant.

We say that an event occurs \emph{asymptotically almost surely (a.a.s.)}~if it occurs with probability that tends to 1 as $n\to \infty$.

\subsection{Related background}

The $t$-stability number of $G_{n,p}$ for the case $t = 0$ has been studied thoroughly for both fixed $p$ and $p(n)=o(1)$.  Matula~\cite{Mat70, Mat72,
 Mat76} and, independently, Grimmett and McDiarmid~\cite{GrMc75} were the first to notice and then prove asymptotic concentration of the stability number
 using the first and second moment methods.  For $0 < p < 1$, define $b := 1/(1-p)$ and
\[ \atp{0} := 2\log_b n -2\log_b\log_b n +2\log_b(e/2)+1. \]
For fixed $0<p<1$, it was shown that for any $\eps > 0$ a.a.s.
\begin{align}
\lfloor \atp{0} - \eps \rfloor \le \alpha_0 (G_{n,p}) \le \lfloor \atp{0} + \eps \rfloor, \label{alpha0}
\end{align}
showing in particular that $\chi(G_{n,p}) \ge (1 - \eps) n/\atp{0}$.  
Assume now that $p=p(n)$ is bounded away from 1. 
Bollob\'as and Erd\H{o}s~\cite{BoEr76} extended~\eqref{alpha0} to hold with $p(n) > n^{-\delta}$ for any $\delta >0$.  
Much later, with the use of martingale techniques, Frieze~\cite{Fri90} showed that for any $\eps > 0$ there exists some 
constant $C_\eps$ such that if $p(n) \ge C_\eps/n$ then~\eqref{alpha0} holds a.a.s.

Efforts to determine the chromatic number of $G_{n,p}$ took place in parallel with the study of the stability number.  For fixed $p$, Grimmett and McDiarmid conjectured that $\chi(G_{n,p}) \sim n/\atp{0}$ a.a.s.  This conjecture was a major open problem in random
 graph theory for over a decade, until Bollob\'as~\cite{Bol88} and Matula and Ku\v{c}era~\cite{MaKu90} used martingales to establish the conjecture.  It was
 crucial for this work to obtain strong upper bounds on the probability of nonexistence in $G_{n,p}$ of a stable set with just slightly fewer than $\atp{0}$
 vertices.  \L{u}czak~\cite{Luc91a} fully extended the result to hold for sparse random graphs; that is, for the case $p(n) = o(1)$ and $p(n) \ge C/n$ for
 some large enough constant $C$.
Consult Bollob\'as~\cite{Bol01} or Janson, {\L}uczak and Ruci\'nski~\cite{JLR00} for a detailed survey of these as well as related results.  

For the case $t \ge 1$, the first results on the $t$-stability number were developed indirectly as a consequence of broader work on hereditary properties of
 random graphs.  A graph property --- that is, an infinite class of graphs closed under isomorphism --- is said to be \emph{hereditary} if every induced subgraph
 of every member of the class is also in the class.  For any given $t$, the class of graphs that are $t$-stable is an hereditary property.  As a result of study in
 this more general context, it was shown by Scheinerman~\cite{Sch92} that, for fixed $p$, there exist constants $c_{p,1}$ and $c_{p,2}$ such that
 $c_{p,1} \ln n \le \s \le c_{p,2} \ln n$ a.a.s.  This was further improved by Bollob\'as and Thomason~\cite{BoTh00} who characterised, for any fixed $p$, an
 explicit constant $c_p$ such that $(1-\eps)c_p \ln n \le \s \le (1+\eps)c_p \ln n$ a.a.s.  For any fixed hereditary property, not just $t$-stability, the constant
 $c_p$ depends upon the property but essentially the same result holds.  Recently, Kang and McDiarmid~\cite{KaMc07,KaMc10} considered $t$-stability separately, but
 also treated the situation in which $t = t(n)$ varies (i.e.~grows) in the order of the random graph.  They showed that, if $t = o(\ln n)$, then a.a.s.
 
\begin{align}
(1-\eps)2 \log_b n \le \s \le (1+\eps)2 \log_b n \label{alphatbasic}
\end{align}
(where $b = 1/(1-p)$, as above).  In particular, observe that the estimation~\eqref{alphatbasic} for $\s$ and the estimation~\eqref{alpha0} for
 $\alpha_0(G_{n,p})$ agree in their first-order terms.  This implies that as long as $t = o(\ln n)$ the $t$-improper and the ordinary chromatic numbers of
 $G_{n,p}$ have roughly the same asymptotic value a.a.s.

\subsection{The results of the present work}
In this paper, we restrict our attention to the case in which the edge probability $p$ and the non-negative integer parameter $t$ are fixed constants.  Restricted
 to this setting, our main theorem is an extension of~\eqref{alpha0} and a strengthening of~\eqref{alphatbasic}.
\begin{theorem}\label{1stability} Fix $0 < p < 1$ and $t \ge 0$.  Set $b:=1/(1-p)$ and 
\[
\atp{t} :=
 2\log_b n
 + (t-2)\log_b\log_b n
 + \log_b (t^t/t!^2)
 + t\log_b(2 b p/e)
 + 2\log_b (e/2)+1.
\]
Then for every $\eps > 0$ a.a.s.
\[ \lfloor \atp{t} - \eps \rfloor \le \s \le \lfloor \atp{t} + \eps \rfloor. \]
\end{theorem}
\noindent
We shall see that this theorem in fact holds if $\eps=\eps(n)$ as long as $\eps \gg \ln \ln n / \sqrt{\ln n}$.

We derive the upper bound with a first moment argument, which is presented in Section~\ref{1stMom}.  To apply the first moment method, we estimate the
 expected number of $t$-stable sets that have order $k$.  In particular, we show the following.
\begin{theorem} \label{Expectation} 
Fix $0 < p < 1$ and $t \ge 0$.  Let $\alpha_t^{(k)}(G)$ denote the number of $t$-stable sets of order $k$ that are contained in a graph $G$. 
If $k=O(\ln n)$ and $k\to \infty$ as $n\to \infty$, then 
\[ \ex (\sk) =
\left( 
e^2 n^2 b^{-k+1} k^{t-2} \left(\frac{t b p}{e}\right)^t \frac{1}{t!^2} \right)^{k/2} (1 + o(1))^k. \]
\end{theorem}
\noindent
(Note that by~\eqref{alphatbasic} the condition on $k$ is not very restrictive.)
Using this formula, we will see in Section~\ref{1stMom} that the expected number of $t$-stable sets with $\lfloor \atp{t} + \eps \rfloor+1$ vertices tends to
 zero as $n\to \infty$.   

The key to the calculation of this expected value is a precise formula for the number of degree sequences on $k$ vertices with a given number of edges and
 maximum degree at most $t$.  In Section~\ref{DegSeq}, we obtain this formula by the inversion formula of generating functions --- applied in our case to the generating function of
 degree sequences on $k$ vertices and maximum degree at most $t$. This formula is an integral of a complex function that is approximated with the use of  an
 analytic technique called saddle-point approximation. Our proof is inspired by the application of this method by Chv\'atal~\cite{Chv91} to a similar generating
 function.  For further examples of the use of the saddle-point method, consult Chapter~VIII of Flajolet and Sedgewick~\cite{FlSe09}.

The lower bound in Theorem~\ref{1stability} is derived with a second moment argument in Section~\ref{2ndMom}.

We remark that Theorems~\ref{1stability} and~\ref{Expectation} are both stated to hold for the case $t = 0$ (if we assume that $0^0 = 1$) in order to stress that these results generalise
 the previous results of Matula~\cite{Mat70, Mat72, Mat76} and Grimmett and McDiarmid~\cite{GrMc75}.  Our methods apply for this special case,
 however in our proofs our main concern will be to establish the results for $t \ge 1$.

In Section~\ref{chrom} we give a quite precise formula for the $t$-improper chromatic number of $G_{n,p}$.  For $t = 0$, that is, for the chromatic
 number, McDiarmid~\cite{McD90} gave a fairly tight estimate on $\chi(G_{n,p})(=\chi_0(G_{n,p}))$
proving that for any fixed $0< p < 1$ a.a.s.
\begin{align*}
\frac{n}{\atp{0} - 1 -o(1)} \le \chi_0(G_{n,p}) \le \frac{n}{\atp{0} -1 -\frac{1}{2} - \frac{1}{1-(1-p)^{1/2}}+o(1)}.
\end{align*}
Panagiotou and Steger~\cite{PaSt09} recently improved the lower bound showing that a.a.s. 
\begin{align*}
\chi_0(G_{n,p})  \ge \frac{n}{\atp{0} - \frac{2}{\ln b} -1 +o(1)},
\end{align*}
and asked if better upper or lower bounds could be developed.
In Section~\ref{chrom}, we improve upon McDiarmid's upper bound and we generalise (for $t \ge 1$) both this new bound and the lower bound of Panagiotou and Steger.
\begin{theorem} \label{tchrom}
Fix $0 < p < 1$ and $t\ge 0$.  Then a.a.s. 
\begin{align*}
\frac{n}{\atp{t} -\frac{2}{\ln b} -1 + o(1)} \le \chi_t(G_{n,p}) \le  \frac{n}{\atp{t} -\frac{2}{\ln b} - 2 - o(1)}.
\end{align*}
\end{theorem}

Given a graph $G$, let the {\em colouring rate} $\overline{\alpha_0}(G)$ of $G$ be
$|V(G)| /\chi_0(G)$, which is the maximum average size of a colour class in a proper colouring of $G$.  Then the case $t=0$ of Theorem~\ref{tchrom} implies for any fixed $0 < p < 1$ that a.a.s.
\begin{align*}
\atp{0} -\frac{2}{\ln b} - 2 - o(1) \le \overline{\alpha_0} (G_{n,p}) \le \atp{0} -\frac{2}{\ln b} - 1 + o(1).
\end{align*}
Thus the colouring rate of $G_{n,p}$ is a.a.s.~contained in an explicit interval of length $1+o(1)$.  We remark that Shamir and Spencer~\cite{ShSp87} showed a.a.s.~$\tilde{O}(\sqrt{n})$-concentration of $\chi_0(G_{n,p})$ --- see also a recent improvement by Scott~\cite{Sco08+}.  (The $\tilde{O}$ notation ignores logarithmic factors.)  It therefore follows that $\overline{\alpha_0}(G_{n,p})$ is a.a.s.~$\tilde{O}( n^{-1/2})$-concentrated. 

The above discussion extends easily to $t$-improper colourings.


\section{Counting degree sequences of maximum degree $t$} \label{DegSeq}

Given non-negative integers $k, t$ with $t < k$, we let 
\[ C_{2 m}(t,k):=\sum_{(d_1,\ldots,d_k),\sum_i d_i =2 m,d_i\le t}\frac{1}{\prod_i d_i!}. \]
(Here, the $d_i$ are non-negative integers.)
Given a fixed degree sequence $(d_1,\ldots, d_k)$ with $\sum_i d_i = 2 m$, the number of graphs on $k$ vertices $(v_1,\ldots , v_k)$ where $v_i$ has
 degree $d_i$ is at most
\begin{align*} 
\frac{1}{\prod_i d_i!} \frac{(2 m)!}{m!2^m}.
\end{align*}
See for example~\cite{Bol01} in the proof of Theorem~2.16 or Section~9.1 in~\cite{JLR00} for the definition of the configuration model, from which the
 above claim follows easily.  Therefore, $C_{2 m}(t,k){(2 m)!/(m!2^m)}$ is an upper bound on the number of graphs with $k$ vertices and 
$m$edges such that each vertex has degree at most $t$. Note also that $(2 m)! C_{2 m}(t,k)$ is the number of allocations of $2m$ balls 
into $k$ bins with the property that no bin contains more than $t$ balls. 

In the proof of Theorem~\ref{Expectation}, we need good estimates for $C_{2m}(t,k)$, when $2m$ is close to $tk$. 
In particular, as we will see in the next section (Lemma~\ref{mstar}) we will need a tight estimate for $C_{2m} (t,k)$ when 
$t - \ln k/\sqrt{k} < 2m/k < t - 1/(\sqrt{k} \ln k)$, since in this range the expected number of $t$-stable sets having 
$m$ edges is maximised.
We require a careful and specific treatment of this estimation due to
the fact that $2m/k$ is not bounded below $t$.

For $t\ge 1$, note that $C_{2 m}(t,k)$ is the coefficient of $z^{2 m}$ in the following generating function: 
\[ G(z)=R_t(z)^{k}= \left(\sum_{i=0}^t \frac{z^i}{i!} \right)^{k}. \]
Cauchy's integral formula gives
\[ C_{2 m}(t,k)= \frac{1}{2\pi i} \int_{C} \frac{R_t(z)^k}{z^{2 m+1}} dz, \]
where the integration is taken over a closed contour containing the origin.

Before we state the main theorem of this section, we need the following lemma, which follows from Note~IV.46 in~\cite{FlSe09}.

\begin{lemma}\label{lemma:r_0,s}
Fix $t\ge 1$.  The function $r R_t'(r)/R_t(r)$ is strictly increasing in $r$ for $r > 0$.  For each $y \in (0, t)$, there exists a unique positive solution 
$r_0 = r_0(y)$ to the equation 
$r R_t'(r)/R_t(r)=y$ and furthermore the function $r_0(y)$ is a continuous bijection between $(0,t)$ and $(0,\infty)$.  Thus, if we set
\[ s(y)=r_0(y) \frac{d}{d x} \left. \frac{x R_t'(x)}{R_t(x)}\right|_{x=r_0(y)},\]
then $s(y) > 0$.
\end{lemma}
We will prove a ``large powers'' theorem to obtain a very tight estimate on $C_{2m}(t,k)$ when $2m/k$ is quite close to $t$.
A version of this theorem holds if we instead assume that $2m/k$ is bounded away from $t$; indeed, this immediately follows from Theorem~VIII.8 of~\cite{FlSe09}.  However, our version, where $2m/k$ approaches $t$, is necessary in light of Lemma~\ref{mstar} below.
\begin{theorem} \label{coupons}
Assume that $t\ge 1$ is fixed and $k \to \infty$.  Suppose that $m$ and $k$ are such that $t - \ln k/\sqrt{k} \le 2 m/k \le t - 1/(\sqrt{k}\ln k)$ for any $\eps > 0$, and $r_0$ and $s$ are defined as in 
Lemma~\ref{lemma:r_0,s}.  Then uniformly
\begin{align*} 
C_{2 m}(t,k)= \frac{1}{\sqrt{2\pi k s(2 m/k)}} \frac{R_t(r_0(2 m/k))^k}{r_0(2 m/k)^{2 m}}(1+o(1)).
\end{align*}
\end{theorem}

In the proof of the theorem (as well as in later sections), we make frequent use of the following lemma, whose proof is postponed until the end of the section.
\begin{lemma}\label{ytot}
If $y=y(k) \to t$ as $k \to \infty$ (and $y<t$) and $r_0$ and $s$ are defined as in Lemma~\ref{lemma:r_0,s}, then
\begin{align}
&r_0
 = \frac{t}{t-y} + O(1),\label{r_asympt}
&\\
&\frac{d r_0}{d y}
 = \frac{{r_0}^2}{t}\left(1+O\left(\frac{1}{r_0}\right)\right), \text{ and} \label{derivative}
&\\
&s
 = \frac{t}{r_0}\left(1+O\left(\frac{1}{r_0}\right)\right). \label{lemma:s}&
\end{align}
\end{lemma}

\begin{proof}{\bf of Theorem~\ref{coupons}}
The proof is inspired by~\cite{Chv91}.  
Throughout, we for convenience drop the subscript and write $R(z)$ in the place of $R_t(z)$.
Recall that $r_0 = r_0(2m/k)$ is the unique positive solution of the equation $rR'(r)/R(r) = 2m/k$, where 
$t- \ln k/\sqrt{k} \le 2m/k \leq t-1/(\sqrt{k} \ln k)$, and 
let $C$ be the circle of radius $r_0$ centred at the origin.  Using polar coordinates, we obtain
\begin{align*}
C_{2 m}(t,k)&=\frac{1}{2\pi i} \int_{C} \frac{R(r_0 e^{i\varphi})^k}
{{r_0}^{2 m+1}e^{i 2 m\varphi}e^{i\varphi}} d(r_0 e^{i\varphi}) 
= \frac{1}{2\pi {r_0}^{2 m}} \int_{-\pi}^{\pi} \frac{R(r_0 e^{i\varphi})^k}
{e^{i 2 m\varphi}} d \varphi.
\end{align*}
We let $\delta = \delta(k) := \ln k \sqrt{r_0/k}$ and write
\begin{align} \label{eqn:C}
C_{2 m}(t,k)=
\frac{1}{2\pi {r_0}^{2 m}} \left(
\int_{\delta}^{2\pi-\delta} \frac{R(r_0 e^{i\varphi})^k}{e^{i 2 m\varphi}} d \varphi +
\int_{-\delta}^{\delta} \frac{R(r_0 e^{i\varphi})^k}{e^{i 2 m\varphi}} d \varphi \right) .
\end{align}
Note that, since $2 m/k < t - 1/(\ln k \sqrt{k})$, it follows from~\eqref{r_asympt} that $\delta \to 0$ as $k \to \infty$.  We shall analyse the two integrals of~\eqref{eqn:C} separately.

To begin, we consider the first integral of~\eqref{eqn:C} and we wish to show that it makes a negligible contribution to the value of $C_{2 m}(t,k)$.  Note that
\begin{align}
\left|R(r_0 e^{i \varphi})\right|^2
&= \left( \sum_{j=0}^{t} \frac{{r_0}^{j}}{j!} \cos(j \varphi)
\right)^2 +
\left( \sum_{j=0}^{t} \frac{{r_0}^{j}}{j!} \sin(j \varphi)
\right)^2 \nonumber\\
&= \sum_{0 \le j_1,j_2 \le t} \frac{{r_0}^{j_1+j_2}}{j_1!j_2! }
\left(\cos (j_1 \varphi) \cos (j_2 \varphi)+
\sin (j_1 \varphi) \sin (j_2 \varphi) \right) \nonumber\\
&= \sum_{0 \le j_1,j_2 \le t} \frac{{r_0}^{j_1+j_2}}{j_1!j_2!}
\cos \left( (j_1-j_2) \varphi \right) \nonumber\\
&= R(r_0)^2
- \sum_{0 \le j_1 < j_2 \le t} \frac{2 {r_0}^{j_1+j_2}}{j_1!j_2!}
\left(1-\cos \left( (j_1-j_2) \varphi \right) \right). \label{eqn:R}
\end{align}

Note that $r_0\to \infty$ as $k \to \infty$. Hence, from~\eqref{eqn:R},
\begin{align*}
\left|R(r_0 e^{i \varphi})\right|^2
&\le R(r_0)^2 \left( 1 - \frac{\frac{2 {r_0}^{2 t-1}}{t!(t-1)!} (1 - \cos \varphi)}{\frac{{r_0}^{2 t}}{t!^2} + \Theta({r_0}^{2 t-1})} \right) 
= R(r_0)^2 \left( 1-(1+o(1))\frac{2 t}{r_0}( 1-\cos \varphi) \right).
\end{align*}
It follows that for $k$ large enough
\begin{align} \label{caseBfirst}
\left| \int_{\delta}^{2\pi-\delta} \frac{R(r_0 e^{i\varphi})^k}{e^{i 2 m\varphi}} d\varphi \right|
& \le 2\pi R(r_0)^k\left( 1-(1+o(1))\frac{2 t}{r_0}( 1-\cos \delta) \right)^{k/2} \nonumber\\
& \le 2\pi R(r_0)^k\exp \left( -\frac{t k}{2r_0}( 1-\cos \delta) \right)  \nonumber\\
& = 2\pi R(r_0)^k\exp \left(-\frac{t}{ 2} \cdot \frac{k\delta^2}{r_0 \ln k} \cdot \frac{1-\cos \delta}{\delta^2} \cdot \ln k \right).
\end{align}
Since $\delta \to 0$, we have that $(1-\cos \delta)/\delta^2 \to 1/2$.  By the choice of $\delta$, we also have that $k\delta^2/(r_0 \ln k) \to \infty$ as $k \to \infty$, and it follows from Inequality~\eqref{caseBfirst} that
\begin{align} \label{first.approx}
\left|\int_{\delta}^{2\pi-\delta} \frac{R(r_0 e^{i\varphi})^k}{e^{i 2 m\varphi}} d\varphi
\right| < R(r_0)^k/k,
\end{align}
for large enough $k$. 

In order to precisely estimate the second integral of~\eqref{eqn:C}, we consider the function $f: \mathbb{R} \to \mathbb{C}$ given by
\begin{align*} 
f(\varphi) &:= R(r_0 e^{i\varphi})\exp \left(-i\frac{2 m}{k} \varphi \right)
= \exp \left(-i\frac{2 m}{k} \varphi \right) 
\left( \sum_{j=0}^{t} \frac{{r_0}^j}{j!} (\cos(j \varphi) + i \sin(j \varphi) )\right).
\end{align*}
The importance of the function $f$ is that
\begin{align*}
\int_{-\delta}^{\delta} \frac{R(r_0e^{i\varphi})^k}{e^{i 2 m\varphi}} d \varphi = \int_{-\delta}^{\delta} f(\varphi)^k d \varphi.
\end{align*}
We will show that the real part of $f(\varphi)^k$ is well approximated by $R(r_0)^k\exp(-s k \varphi^2/2)$ when $|\varphi|$ is small --- see~\eqref{???} below.  The imaginary part can be ignored as the integral approximates a real quantity.

To this end we will apply Taylor's Theorem, 
and in order to do this we shall need the first, second and third derivatives of $f$ with respect to $\varphi$.  First,
\begin{align*}
f'(\varphi) = 
 \exp \left(-i\frac{2 m}{k} \varphi \right)
\left( \sum_{j=0}^{t} \frac{{r_0}^j}{j!} \left(\frac{2 m}{k}-j
 \right)(\sin (j \varphi) - i \cos (j \varphi ) )\right). 
\end{align*}
Note that
\begin{align*}
f'(0) &= -i
\left( \frac{2 m}{k}
\sum_{j=0}^{t} \frac{{r_0}^{j}}{j!} - 
\sum_{j=0}^{t} \frac{{r_0}^{j}}{j!}j  \right)
=-i \left( \frac{2 m}{k}R(r_0) - r_0R'(r_0) \right)= 0
\end{align*}
by the choice of $r_0$.
Next,
\begin{align*}
f''(\varphi) 
=& -i \frac{2 m}{k} f'(\varphi)
 +\exp \left(-i \frac{2 m}{k} \varphi \right)
\left( \sum_{j=0}^{t} \frac{{r_0}^j}{j!} \left(\frac{2 m}{k}-j
 \right)j(\cos (j \varphi )+i\sin (j \varphi) )\right).
\end{align*}
Therefore,
\begin{align}
f''(0)&=
-i\frac{2 m}{k}f'(0)
+\sum_{j=0}^{t} \frac{{r_0}^j}{j!} \left(\frac{2 m}{k}-j \right)j \nonumber \\
&= \frac{2 m}{k} \sum_{j=1}^{t} \frac{{r_0}^j}{j!}j - \sum_{j=1}^{t} \frac{{r_0}^j}{j!}j(j-1) - \sum_{j=1}^{t} \frac{{r_0}^j}{j!}j \nonumber \\
&= \left(\frac{r_0 R'(r_0)}{R(r_0)}\right) r_0 R'(r_0) - {r_0}^2 R''(r_0) - r_0 R'(r_0) \nonumber \\
&= -r_0 \left(\frac{-r_0 R'(r_0)^2}{R(r_0)} + r_0 R''(r_0) + R'(r_0) \right) \nonumber \\
&= -R(r_0) r_0 \left(\frac{(r_0 R''(r_0) + R'(r_0))R(r_0)-r_0 R'(r_0)^2}{R(r_0)^2}\right) \nonumber \\
&= 
-R(r_0)r_0 \frac{d}{d x} \left.\frac{xR'(x)}{R(x)}\right|_{x=r_0} 
= -R(r_0)s(2 m/k). \label{sec.par}
\end{align}
Thus, $f''(0) < 0$ by Lemma~\ref{lemma:r_0,s}.  Last, we have
\begin{align*}
f'''(\varphi) 
=& - i \frac{2 m}{k} f''(\varphi)
 - i \frac{2 m}{k} \exp \left(-i \frac{2 m}{k} \varphi \right)
\left( \sum_{j=0}^{t} \frac{{r_0}^j}{j!} \left(\frac{2 m}{k}-j
 \right)j(\cos (j \varphi )+i\sin (j \varphi) )\right)\\
& + \exp \left(-i \frac{2 m}{k} \varphi \right)
\left( \sum_{j=0}^{t} \frac{{r_0}^j}{j!} \left(\frac{2 m}{k}-j
 \right)j^2(-\sin (j \varphi )+i\cos (j \varphi) )\right).
\end{align*}

Since $r_0\rightarrow \infty$ as $k \rightarrow \infty$, there is a positive constant $a$ such that $a \le r_0$, for $k$ sufficiently large.
Clearly, $f(0) = R(r_0) > a^t/t! > 0$.  The continuity of $f$ on the compact set $-\pi \le \varphi \le \pi$ implies that there is a positive constant $\delta_0$ such that whenever $|\varphi| \le \delta_0$ we have $Re (f(\varphi)) > 0$. 
Since the first two derivatives of $Im (f(\varphi))$ with respect to $\varphi$ vanish when $\varphi =0$, and also $Im (f(0))=0$, Taylor's Theorem implies that 
\begin{align*}
|Im (f(\varphi))| \le  \sup_{|\varphi| \le \delta_0} |Im(f'''(\varphi))|\frac{\varphi^3}{6}
\end{align*}
if $|\varphi| \le \delta_0$.
Now, note that $Re (f (\varphi))$ and $Im(f'''(\varphi))$ can be considered as polynomials of degree $t$ with respect to $r_0$.  The leading term of $Re (f (\varphi))$ is
\begin{align*}
Re\left( \exp \left(-i\frac{2 m}{k} \varphi \right) (\cos (t \varphi) + i \sin (t \varphi)) \right) \frac{{r_0}^t}{t!};
\end{align*}
thus, 
$Re (f (\varphi)) = \Omega({r_0}^t)$.
On the other hand, using the derivative computations above and simplifying, it follows that the leading term of $Im(f'''(\varphi))$ is
\begin{align*}
&Im\left( \exp \left(-i\frac{2 m}{k} \varphi \right) (\sin (t \varphi) + i \cos (t \varphi)) \right) \left(t - \frac{2 m}{k} \right)^3 \frac{{r_0}^t}{t!}.
\end{align*}
By~\eqref{r_asympt}, $t - 2 m/k = (1 + o(1)) t / r_0$ and thus $Im(f'''(\varphi)) = O({r_0}^{t-1})$.
So, there exists $c_1 > 0$ such that for every $\varphi$ with $|\varphi|\le \delta_0$
\[ \frac{\sup_{|\varphi |\le \delta_0} |Im(f'''(\varphi))|}{|Re(f(\varphi))|} < \frac{c_1}{r_0}, \]
and therefore
\[\left| \frac{Im (f(\varphi))}{Re (f(\varphi))}\right| \le \frac{c_1
\varphi^3}{6 r_0}, \]
for any $\varphi$ with $|\varphi| \le \delta_0$. 
On the other hand, we have (see pages~15--16 of~\cite{Chv91} for the details) 
\[\left| \frac{Re(z^k)}{Re(z)^k}-1\right| \le \epsilon \left(k,
\left|\frac{Im (z)}{Re (z)} \right| \right), \]
with 
\[\epsilon(k,x)=(1+x)^k-1-x k \le e^{x k} - 1 \]
(for $x \ge 0$).  Since $\epsilon(k,x)$ increases in $x$ for $x\ge 0$, we have 
\begin{align} \label{eqn:f^k}
1- \epsilon \left(k,\frac{c_1 \delta^3}{6 r_0}\right) \le   
\frac{Re ( f(\varphi)^k)}{Re ( f(\varphi))^k}
\le 1+ \epsilon \left(k,\frac{c_1 \delta^3}{6 r_0}  \right),
\end{align}
whenever $|\varphi| \le \delta \le \delta_0$. 
 
Next, we approximate the function $\ln Re (f(\varphi))$.  First,
\begin{align*}
\left. \frac{d}{d \varphi} (\ln Re (f(\varphi))) \right|_{\varphi=0} = \left. \frac{Re (f'(\varphi))}{Re (f(\varphi))} \right|_{\varphi=0} = 0.
\end{align*}
Second, we have
\begin{align*}
\frac{d^2}{d \varphi^2} (\ln Re (f(\varphi)))
& = \frac{d}{d \varphi} \left( \frac{Re (f'(\varphi))}{Re (f(\varphi))} \right) 
= \frac{Re(f''(\varphi))Re(f(\varphi))-Re(f'(\varphi))^2}{Re(f(\varphi))^2};
\end{align*}
therefore, by Equation~\eqref{sec.par},
\begin{align*}
\left. \frac{d^2}{d \varphi^2} (\ln Re (f(\varphi))) \right|_{\varphi=0}
& = \frac{Re(f''(0))Re(f(0))-Re(f'(0))^2}{Re(f(0))^2}
= \frac{-R(r_0) s}{R(r_0)} = -s
\end{align*}
Now, the numerator of the third derivative with respect to $\varphi$ is 
\begin{align*}
& (Re(f''(\varphi))Re(f(\varphi))-Re(f'(\varphi))^2)'Re(f(\varphi))^2  \\
&- 2 Re(f(\varphi))(Re(f''(\varphi))Re(f(\varphi))-Re(f'(\varphi))^2) \\
&= Re (f(\varphi))\Big((Re(f''(\varphi))Re(f(\varphi))-Re(f'(\varphi))^2)'Re(f(\varphi)) \\
& - 2 (Re(f''(\varphi))Re(f(\varphi))-Re(f'(\varphi))^2) \Big). 
\end{align*}
Thus an elementary calculation gives that (for $|\varphi| \le \delta_0$)
\begin{align*}
&\frac{d^3}{d \varphi^3} (\ln Re(f(\varphi))) \\
&= \frac{Re(f'''(\varphi))Re(f(\varphi))^2-3 Re(f''(\varphi))Re(f'(\varphi))Re(f(\varphi))+2 Re(f'(\varphi))^3}{Re(f(\varphi))^3}.
\end{align*}
If, as we did earlier for $Re(f(\varphi))$ and $Im(f'''(\varphi))$, we consider $Re(f'(\varphi))$, $Re(f''(\varphi))$ and $Re(f'''(\varphi))$ as polynomials with respect to $r_0$, we can show that $Re(f'(\varphi)) = O({r_0}^{t-1})$, $Re(f''(\varphi)) = O({r_0}^{t-1})$ and $Re(f'''(\varphi)) = O({r_0}^{t-1})$.
It then follows that there exists $c_2 > 0$ such that for every $\varphi$ with $|\varphi| \le \delta_0$
\begin{align*}
\left|\frac{d^3}{d \varphi^3} (\ln Re (f(\varphi)))\right| \le \frac{c_2}{r_0}.
\end{align*} 
Therefore, Taylor's Theorem implies that for every $\varphi$ with $|\varphi| \le \delta_0$ we have 
\[\left| \ln Re(f(\varphi))- \left(\ln R(r_0) - \frac{s\varphi^2}{2}
\right)\right| \le 
\frac{c_2 \varphi^3}{6 r_0}.  \]
It follows that 
\[ \exp \left(- \frac{c_2 k \delta^3}{6 r_0}\right) \le 
\frac{Re (f(\varphi))^k}{R(r_0)^k \exp (-s k\varphi^2/2)}
\le \exp \left( \frac{c_2 k \delta^3}{6 r_0}\right).\]
The condition that $2 m/k < t - 1/(\ln k\sqrt{k})$ and~\eqref{r_asympt} together imply that $r_0 < t \ln k\sqrt{k} + O(1)$.  Therefore, $k\delta^3/r_0 = \sqrt{r_0/k}\ln^3 k \to 0$ as $k
\to \infty$, and we have 
\[ 
\exp \left( \frac{c_2 k \delta^3}{6 r_0}\right)=1+o(1)
\  \text{ and } \
\epsilon \left(k,\frac{c_1 \delta^3}{6 r_0}  \right)\le \exp\left(\frac{c_1 k\delta^3}{6 r_0} \right)-1= o(1),
\]
proving that 
\begin{align}
Re (f(\varphi)^k)= R(r_0)^k \exp (-s k\varphi^2/2) (1+o(1)) \label{???}
\end{align}
uniformly for $|\varphi| \le \delta$.
From~\eqref{eqn:C},~\eqref{first.approx} and~\eqref{???}, we obtain
\begin{align}\label{eqn:approx}
2\pi{r_0}^{2 m} C_{2 m}(t,k)
&= R(r_0)^k\left(\int_{-\delta}^{\delta} \exp (-s k\varphi^2/2) d
\varphi+o(1)\right).
\end{align} 
Using a change of variables $\psi =\sqrt{s k}\varphi$, observe that 
\[\int_{-\delta}^{\delta}  \exp \left(-\frac{s k\varphi^2}{2} \right) d
\varphi = \frac{1}{\sqrt{s k}} \int_{-\delta\sqrt{s k}}^{\delta\sqrt{s k}} \exp
\left(-\frac{\psi^2}{2} \right) d \psi = \sqrt{\frac{2\pi}{s k}}(1 + o(1)), \]
as $k \to \infty$ since $\delta \sqrt{s k} \sim \sqrt{t}\ln k \to \infty$.
Thus, Equation~\eqref{eqn:approx} becomes
\[
2\pi{r_0}^{2 m} C_{2 m}(t,k)
=\sqrt{ \frac{2\pi}{ks}} R(r_0)^k(1+o(1))
\]
and the result follows. 
\end{proof}

\subsection{Proof of Lemma~\ref{ytot}}

\begin{proof}{\bf of Equation~\eqref{r_asympt}}
First, note that $r_0 = r_0(y)\to \infty$ as $k\to \infty$ by Lemma~\ref{lemma:r_0,s}. So
\begin{align*} 
r_0 R'(r_0) & = \frac{{r_0}^t}{(t-1)!}\left(1+ \frac{t-1}{r_0} +
O\left(\frac{1}{{r_0}^2}\right) \right), \\
R(r_0) & = \frac{{r_0}^t}{t!}\left(1+\frac{t}{r_0} + O\left(\frac{1}{{r_0}^2} \right) \right).
\end{align*}
Thus, 
\begin{align} \label{eq:expand}
\frac{r_0 R'(r_0)}{R(r_0)} &= t\frac{1+\frac{t-1}{r_0} + O\left(\frac{1}{{r_0}^2} \right)}{1+\frac{t}{r_0} + O\left(\frac{1}{{r_0}^2} \right)} 
= t\left(1+\frac{t-1}{r_0} + O\left(\frac{1}{{r_0}^2} \right) \right)\left( 1-\frac{t}{r_0} + O\left(\frac{1}{{r_0}^2} \right) \right) 
\nonumber \\
&= t\left(1-\frac{t}{r_0}+\frac{t-1}{r_0} +  O\left(\frac{1}{{r_0}^2} \right)\right)
= t\left(1-\frac{1}{r_0}+ O\left(\frac{1}{{r_0}^2} \right) \right).
\end{align}

Since $r_0 R'(r_0)/R(r_0) = y = t(1-(t-y)/t)$, we obtain
\begin{align} \label{eq:Reltoy}
1 - \frac{t-y}{t} = 1 - \frac{1}{r_0} + O\left(\frac{1}{{r_0}^2} \right)
\end{align}
which can be rewritten as 
\[ r_0 = \frac{t}{t-y}\left(1 + O\left(\frac{1}{r_0} \right)\right), \]
and this implies the desired expression.
\end{proof}

\begin{proof}{\bf of Equation~\eqref{derivative}}
A more careful treatment of the computations for the proof of~\eqref{r_asympt} shows that the $O(1/{r_0}^2)$ error term in~\eqref{eq:Reltoy} may instead be written $\eta(1/r_0)/{r_0}^2$ where $\eta$ is a power series with positive radius of convergence.  
In particular, as $r_0R'(r_0)$ and $R(r_0)$ are polynomial functions of $r_0$, \eqref{eq:expand} yields, for some power series 
$\eta_1$, $\eta_2$ and $\hat{\eta}_2$ with positive radius of convergence, that
\begin{align*}
\frac{y}{t} &= \frac{r_0 R'(r_0)}{t R(r_0)}
= \frac{1+\frac{t-1}{r_0} + \frac{\eta_1(1/r_0)}{{r_0}^2}}{1+\frac{t}{r_0} + \frac{\eta_2(1/r_0)}{{r_0}^2}}
= \left( 1+\frac{t-1}{r_0} + \frac{\eta_1(1/r_0)}{{r_0}^2}\right) 
\left(1-\frac{t}{r_0} + \frac{\hat{\eta}_2(1/r_0)}{{r_0}^2} \right) \\
& = 1 - \frac{1}{r_0} + \eta\left(\frac{1}{r_0} \right)\frac{1}{{r_0}^2}.
\end{align*}
Then, by differentiating both sides of this expression with respect to $y$, we obtain 
\begin{align*}
\frac{1}{ t} = \frac{d }{ d r_0} \left(1 - \frac{1}{r_0} + \eta\left(\frac{1}{r_0} \right)\frac{1}{{r_0}^2} \right) \frac{d r_0}{d y}.
\end{align*}
We have that
\begin{align*}
\frac{d }{ d r_0} \left(1 - \frac{1}{r_0} + \eta\left(\frac{1}{r_0} \right)\frac{1}{{r_0}^2} \right)
& = \frac{1}{ {r_0}^2} - \eta\left(\frac{1}{r_0} \right)\frac{2}{{r_0}^3} - \eta'\left(\frac{1}{r_0} \right)\frac{1}{{r_0}^4}
= \frac{1}{ {r_0}^2} + O\left( \frac{1}{ {r_0}^3} \right)
\end{align*}
and~\eqref{derivative} immediately follows.
\end{proof}

\begin{proof}{\bf of Equation~\eqref{lemma:s}}
By the definition of $r_0$, it follows from the chain rule that
\[ 1=\frac{d}{d y}\frac{r_0R'(r_0)}{R(r_0)} = \frac{d}{d r_0} \frac{r_0R'(r_0)}{R(r_0)} \frac{d r_0}{d y}.\]
Thus, 
\begin{align*}
\left.\frac{d}{d x}\frac{x R'(x)}{R(x)}\right|_{x=r_0(y)} = \left(\left.\frac{d r_0(y')}{d y'}\right|_{y' = y} \right)^{-1},
\end{align*}
implying that
\begin{align*}
s(y) &= r_0(y)\left(\left.\frac{d r_0(y')}{d y'}\right|_{y'=y} \right)^{-1}
\stackrel{\mathtt{\eqref{derivative}}}{=} 
\frac{t}{r_0(y)} \left(1+O\left(\frac{1}{r_0(y)}\right)\right)
\end{align*}
as required.
\end{proof}


\section{The expected number of $t$-stable sets of order $k$ - proof of Theorem~\ref{Expectation}}\label{1stMom}

In this section, we give an asymptotic expression for the expected number of $t$-stable subsets of $V_n$ of order $k$ in $G_{n,p}$, proving
Theorem~\ref{Expectation}. In light of~\eqref{alphatbasic}, we will consider $k$ such that $k=k(n)=O(\ln n)$ and $k\to \infty$ as $n\to \infty$.
Towards the end of the section, we will specify $k$ and derive the upper bound of Theorem~\ref{1stability} by a first moment argument. 
  
Let $A$ be a subset of $V_n$ that has order $k$. If $\sk$ denotes the number of subsets of $V_n$ of order $k$ that are $t$-stable, then
\begin{align*} 
\ex (\sk) = \binom{n}{k} \prob (A \in S_t).  
\end{align*}
Partitioning according to the number of edges that $A$ induces, we have
\begin{align} \label{expansion} 
\prob (A \in S_t) = \sum_{m=0}^{\lfloor t k/2 \rfloor} \prob(A \in S_t, \; e(A)=m).
\end{align} 
By the definition of $C_{2 m}(t,k)$ (given at the beginning of Section~\ref{DegSeq}), it follows that
\begin{align} \label{EdgesExpansion}
\prob(A \in S_t, \; e(A)=m) \le p^m (1-p)^{\binom{k}{2} -m}
C_{2 m}(t,k)\frac{(2 m)!}{m! 2^m}=:f(m).
\end{align}
First, we find the value of $m$ for which the expression $f(m)$ on the right-hand side of~\eqref{EdgesExpansion} is maximised. 
If $m^*$ is such that $f(m^*) = \max \{f(m):  0 \le 2 m \le t k \}$, it turns out
that the following holds.
\begin{lemma} \label{mstar}
\( 2 m^* = t k -\sqrt{t k/b p}+o(\sqrt{k}). \)
\end{lemma}
\begin{proof}
Let $\la_m = \la_m (t,k) =f(m+1)/f(m)$. Thus, 
\begin{align*}
\la_m &= \frac{p}{1-p}\frac{C_{2 m+2}(t,k)}{C_{2 m}(t,k)}\frac{1}{2}
\frac{(2 m+2)(2 m+1)}{m+1}=\frac{p}{1-p}\frac{C_{2 m+2}(t,k)}{C_{2 m}(t,k)}(2 m+1).
\end{align*}
We will estimate $\la_m$ for all $m$ with $0 \le 2 m \le t k$ and treat three separate cases:
\begin{enumerate}
\item\label{1stcase} $2 m < t k - \sqrt{k} \ln k$;
\item\label{2ndcase} $2 m > t k - \sqrt{k} / \ln k$; and
\item\label{3rdcase} $t k - \sqrt{k} \ln k \le 2 m \le t k - \sqrt{k} / \ln k$.
\end{enumerate}
We will use Theorem~\ref{coupons} in Case~\ref{3rdcase}, as we will determine those values $m$ for which $\la_m \approx 1$ within that range.  
In the other cases we will use a cruder argument, which is nonetheless sufficient for our purposes. 

\subsubsection*{Case~\ref{1stcase}} 

We will show that $\la_m > 1$ for any such $m$. 
We set $S_{2 m}(t,k)=(2 m)!C_{2 m}(t,k)$. Note that this is equal to the number of ways of allocating $2 m$ labelled balls into $k$ bins so that each bin does not receive more than $t$ balls --- we also denote the set of such allocations by ${\mathcal S}_{2 m}(t,k)$. We have 
\begin{align} \label{eq:CoeffRatio}
\frac{C_{2(m+1)}(t,k)}{C_{2 m} (t,k)} = \frac{S_{2(m+1)}(t,k)}{S_{2 m} (t,k)}\frac{1}{(2 m+2)(2 m+1)}. 
\end{align}
We will obtain a lower bound on the left-hand side, by first obtaining a lower bound on the ratio $S_{2(m+1)}(t,k)/S_{2 m} (t,k)$. 
Let us consider $2 m+2$ distinct balls which we label $1,\ldots, 2 m+1, 2 m+2$.  We construct an auxiliary bipartite graph whose parts are 
${\mathcal S}_{2 m}(t,k)$ and ${\mathcal S}_{2 m+2}(t,k)$. If $c\in {\mathcal S}_{2 m}(t,k)$ and $c' \in {\mathcal S}_{2 m+2}(t,k)$, 
then $(c,c')$ forms an edge in the auxiliary graph if $c'$ restricted to balls $1,\ldots, 2 m$ is $c$.  So any $c' \in {\mathcal S}_{2 m+2}(t,k)$ is adjacent 
to exactly one configuration $c \in {\mathcal S}_{2 m}(t,k)$, that is, its degree in the auxiliary graph is equal to 1.  
Also, if $e(c)$ is the number of non-full bins in a configuration $c \in {\mathcal S}_{2 m}(t,k)$, then $c$ has at least $e(c)(e(c) - 1)$ neighbours in 
${\mathcal S}_{2 m+2} (t,k)$. This is the case since there are at least $e(c)(e(c) - 1)$ ways of allocating balls $2 m+1$ and $2 m+2$ into the 
non-full bins of $c$,
therefore giving a lower bound on the number of configurations in ${\mathcal S}_{2 m+2} (t,k)$ whose restriction on the first $2 m$ balls is $c$. But $2 m < t k - \sqrt{k} \ln k$ and therefore $e(c)\ge \sqrt{k} (\ln k) / t$.  These observations imply that for $k$ large 
enough
\[ S_{2 m+2} (t,k) \ge \frac{k \ln^2 k}{2 t^2} S_{2 m}(t,k), \]
and therefore 
\begin{align*} 
\frac{C_{2(m+1)}(t,k)}{C_{2 m} (t,k)} &= \frac{S_{2(m+1)}(t,k)}{S_{2 m} (t,k)}\frac{1}{(2 m+2)(2 m+1)}\ge 
\frac{k \ln^2 k}{2(2 m+2)(2 m+1)} = \Omega\left(\frac{\ln^2 k}{ m} \right).   
\end{align*}
So $\la_m = \Omega (\ln^2 k) > 1$ in Case~\ref{1stcase}.
\medskip

\subsubsection*{Case~\ref{2ndcase}}
We treat this case similarly. We consider an auxiliary bipartite graph as above. 
Let  $c \in {\mathcal S}_{2 m}(t,k)$ be a configuration of balls $1,\ldots, 2 m$. 
Since there are at most $\sqrt{k}/\ln k$ places available in the non-full bins,  
there are at most $k/\ln^2 k$ ways of allocating balls $2 m+1$ and $2 m+2$
into the non-full bins of $c$. In other words, the degree of any vertex in ${\mathcal S}_{2 m}(t,k)$ is at most $k/\ln^2 k$. 
Also, as above, the degree of any vertex/configuration  $c' \in {\mathcal S}_{2 m+2}(t,k)$ is equal to one. Therefore, 
\[ \frac{S_{2 m+2}(t,k) }{ S_{2 m}(t,k)} \le  \frac{k}{ \ln^2 k}. \]
Substituting this into~\eqref{eq:CoeffRatio}, we obtain
\[ \frac{C_{2(m+1)}(t,k) }{ C_{2 m} (t,k)} \le  \frac{k }{ \ln^2 k}\frac{1}{ (2 m+2)(2 m+1)}.\]
Therefore, in Case~\ref{2ndcase} we have
\begin{align*}
\la_m = O\left(\frac{k }{ m \ln^2 k}\right) = O\left(\frac{1}{ \ln^2 k}\right) =o(1).
\end{align*}

\subsubsection*{Case~\ref{3rdcase}}

In this range, we need more accurate estimates, as we will identify those $m$ for which $\la_m$ is approximately equal to 1. 
We appeal to Theorem~\ref{coupons} for asymptotic estimates of $C_{2 m}(t,k)$ and $C_{2 m+2}(t,k)$ and write $\la_m = (1 + o(1)) {\tilde \la}_m$ where
\begin{align} \label{ratio}
{\tilde \la}_m &=  \frac{p}{1-p}\left(\frac{s(2 m/k)}{s(2(m+1)/k)}\right)^{1/2}
\left(\frac{R(r_0(2(m+1)/k))}{R(r_0(2 m/k))} \right)^k \frac{r_0(2 m/k)^{2 m}}{r_0(2(m+1)/k)^{2 m+2}}(2 m+1). \nonumber \\
& 
\end{align}

Writing $2 m=t k - x k$, we have  $x = o(1)$. 
So, by~\eqref{r_asympt} and~\eqref{derivative}, uniformly for every $z \in [t-x, t-x+2/k]$, 
we have 
\[
\left. \frac{d r_0}{d y}\right|_{y=z} = \frac{t}{x^2}(1+o(1));
\]
thus, the Mean Value Theorem yields 
\begin{align}
r_0(2(m+1)/k)
& = r_0(2 m/k) + \frac{2t}{x^2 k}(1+o(1)) 
 \stackrel{\text{\eqref{r_asympt}}}{=} r_0(2 m/k)\left(1 + \frac{2}{x k}(1+o(1))\right).  \label{perturb1}
\end{align}
So, since $x k \to \infty$ as $k \to \infty$, Equation~\eqref{perturb1} and~\eqref{lemma:s} yield
\begin{align} \label{1st_rat}
\left(\frac{s(2 m/k)}{s(2(m+1)/k)}\right)^{1/2} = 1+o(1).
\end{align}

To estimate the third ratio of~\eqref{ratio}, we write $r_0(2(m+1)/k) = r_0(2 m/k)(1 + \eta)$ where $\eta = (2/x k)(1+o(1))$ by~\eqref{perturb1}.  We also write
\begin{align*}
R(r_0(2(m+1)/k)
& = \frac{{r_0}^t(2(m+1)/k)}{t!} \sum_{t=0}^t \frac{t!}{(t-\ell)!} \frac{1}{{r_0}^{\ell}(2(m+1)/k)}.
\end{align*}
Note that
\begin{align*}
\sum_{t=0}^t &\frac{t!}{(t-\ell)!} \frac{1}{{r_0}^{\ell}(2(m+1)/k)}
 = \sum_{t=0}^t \frac{t!}{(t-\ell)!} \frac{(1 + \eta)^{-\ell}}{{r_0}^{\ell}(2 m/k)}
 = \sum_{t=0}^t \frac{t!}{(t-\ell)!} \frac{1 - \ell \eta(1 + O(\eta))}{{r_0}^{\ell}(2 m/k)} \\
& = 1 + \frac{t}{r_0(2 m/k)} (1 - \eta) + \frac{t(t-1)}{{r_0}^2(2 m/k)} + O\left( \frac{\eta^2}{r_0(2 m/k)} +  \frac{\eta}{{r_0}^2(2 m/k)} + \frac{1}{{r_0}^3(2 m/k)}\right).
\end{align*}
Since this last big-O term is $o(1/k)$, it follows that
\begin{align*}
\frac{R(r_0(2(m+1)/k)}{r_0(2(m+1)/k)^t}
& = \frac{1}{t!} \left(1 + \frac{t}{r_0(2 m/k)} (1 - \eta) + \frac{t(t-1)}{{r_0}^2(2 m/k)} + o(1/k)\right)
\end{align*}
and similar calculations show that
\begin{align*}
\frac{R(r_0(2 m/k)}{r_0(2 m/k)^t}
& = \frac{1}{t!} \left(1 + \frac{t}{r_0(2 m/k)} + \frac{t(t-1)}{{r_0}^2(2 m/k)} + o(1/k)\right).
\end{align*}
So the third ratio in~\eqref{ratio} becomes
\begin{align}
\left(\frac{R(r_0(2(m+1)/k))}{R(r_0(2 m/k))} \right)^k
& = \left(\frac{r_0(2(m+1)/k)}{r_0(2 m/k)} \right)^{t k}
\left( 1 - \frac{t \eta}{r_0(2 m/k)} + o(1/k) \right)^k \nonumber\\
& = \left(\frac{r_0(2(m+1)/k)}{r_0(2 m/k)} \right)^{t k}
e^{-2} (1 + o(1)) \label{rat2}
\end{align}
where the last equality holds by the fact that
\begin{align*}
\frac{t \eta k}{r_0(2 m/k)} = \frac{t (2/x k) k}{t/x} (1 + o(1)) = 2(1 + o(1)).
\end{align*}
Since $x k\to \infty$, we have by~\eqref{perturb1} and~\eqref{r_asympt} that $r_0(2(m+1)/k)=r_0(2 m/k)(1+o(1)) = (1+o(1))t/x$.  
So using~\eqref{perturb1} and~\eqref{rat2} we can write the product of the third and the fourth terms in~\eqref{ratio} as follows:
\begin{align*}
\left(\frac{R(r_0(2(m+1)/k))}{R(r_0(2 m/k))} \right)^k 
&\frac{{r_0}^{2 m}(2 m/k)}{{r_0}^{2 m+2}(2(m+1)/k)} \\
& = e^{-2} \left(\frac{r_0(2(m+1)/k)}{r_0(2 m/k)} \right)^{t k - 2 m} \frac{1 + o(1)}{{r_0}^2(2(m+1)/k)}
\\
& = e^{-2} 
\left(1 + \frac{2}{x k}(1+o(1))\right)^{x k}
\frac{x^2}{t^2}(1+o(1))
 \stackrel{x k \to \infty}{=} \frac{x^2}{t^2}(1+o(1)). 
\end{align*}
If $x\ge \omega(k)/\sqrt{k}$, where $\omega(k)\to \infty$, then 
substituting this last equation and~\eqref{1st_rat} into~\eqref{ratio} and recalling that $\la_m = (1 + o(1)){\tilde \la}_m$, we obtain
\begin{align*} 
\lambda_m = 
\Omega(1)\frac{x^2}{t^2} (2 m+1) = \Omega\left(\frac{\omega(k)^2 m}{k} \right) = \Omega(\omega(k)^2) \to \infty.
\end{align*}
If $x\le 1/(\omega(k) \sqrt{k})$, then these substitutions yield
\[ \lambda_m = O(1) \frac{x^2}{t^2} (2 m+1) = O\left( \frac{m}{\omega(k)^2 k} \right)= O\left(\frac{1}{\omega(k)^2}\right) = o(1). \]
Assume now that $x=\alpha/\sqrt{k}$, for some $\alpha = \Theta(1)$. In this case, 
\[ \lambda_m = \frac{p}{1-p}\frac{\alpha^2}{t^2 k}(t k -x k + 1) (1+o(1)) = \frac{p}{1-p}\frac{\alpha^2}{t}(1+o(1))\stackrel{b=1/(1-p)}{=} 
\frac{b p \alpha^2}{t}(1+o(1)). \]
Thus for any fixed $\alpha' < \sqrt{t / b p}<\alpha'' $ and for $k$ large enough we have  $tk - \alpha'' \sqrt{k} \leq  2 m^* \leq   t k - \alpha'\sqrt{k}$. 
Putting all these different cases together, we deduce that, if $m^*$ is such that 
$f(m^*)$ is maximised over the set $0 \le 2 m \le t k$, then  
\( 2 m^* = t k  - \sqrt{t k/b p} +o(\sqrt{k}).\)
This concludes the proof of Lemma~\ref{mstar}.
\end{proof}

Before we proceed to the proof of Theorem~\ref{Expectation}, let us use Lemma~\ref{mstar} to compute a precise asymptotic expression for $f(m^*)$.  Recall that $b=1/(1-p)$ and observe that
\begin{align}
p^{m^*} (1-p)^{\binom{k}{2} - m^*}
& = b^{-\binom{k}{2}} (b p)^{t k/2 - \sqrt{t k/b p} + o(\sqrt{k})}  = b^{-\binom{k}{2}} (b p)^{t k/2} \left(1 + O\left(\frac{1}{\sqrt{k}}\right)\right)^k.
\label{pmstar}
\end{align}

For the second part of the expression for $f(m^*)$, note that, by Theorem~\ref{coupons},
\begin{align} \label{C_expr}
C_{2 m^*}(t,k) = 
\frac{1}{\sqrt{2\pi s(2 m^*/k)}}\frac{R(r_0(2 m^*/k))^k}{r_0(2 m^*/k)^{2 m^*}} (1+o(1)).
\end{align}
By~\eqref{r_asympt}, we have 
\begin{align*} 
r_0(2 m^*/k) = \sqrt{t b p k} + o(\sqrt{k}). 
\end{align*}
Thus, by~\eqref{lemma:s}, $s(2 m^*/k)=\Theta (1/\sqrt{k})$.
Now, it follows that 
\begin{align*}
R(r_0(2 m^*/k)) 
& = \frac{{r_0}^t(2 m^*/k)}{t!} \sum_{\ell=0}^{t} \frac{t!}{(t-\ell)!} \frac{1}{{r_0}^\ell(2 m^*/k)} \\
& = \frac{{r_0}^t(2 m^*/k)}{t!} \left( 1 + \frac{t}{r_0(2 m^*/k)} + O\left(\frac{1}{{r_0}^2(2 m^*/k)}\right) \right) \\
& = \frac{{r_0}^t(2 m^*/k)}{t!} \left( 1 + \sqrt{\frac{t}{b p k}} + o\left(\frac{1}{\sqrt{k}}\right) \right);
\end{align*}
therefore,
\begin{align*}
R(r_0(2 m^*/k))^k
& = \frac{(r_0(2 m/k))^{t k}}{t!^k} e^{\sqrt{t k/b p}+o(\sqrt{k})}.
\end{align*}
Substituting this into~\eqref{C_expr}, we obtain
\begin{align}
C_{2 m^*}(t,k) &= \Theta(k^{1/4}) \frac{(r_0(2 m^*/k))^{t k - 2 m^*}}{t!^k} e^{\sqrt{t k/b p}+o(\sqrt{k})}
\nonumber \\
& = \Theta(k^{1/4}) \left(\sqrt{t b p k} + o(\sqrt{k}) \right)^{\sqrt{t k/b p} + o(\sqrt{k})} e^{\sqrt{t k/b p}+ o(\sqrt{k})} \frac{1}{t!^k}
\nonumber\\
& = \frac{1}{t!^k} \left(1 + O\left(\frac{\ln k}{\sqrt{k}}\right)\right)^k. \label{Coef}
\end{align}

For the last part of the expression for $f(m^*)$, we apply Stirling's formula to obtain
\begin{align}
\frac{(2 m^*)!}{m^*! 2^{m^*}}
& = \frac{(2 m^*/e)^{2 m^*} \sqrt{2 \pi (2 m^*)} e^{o(1)}}{(m^*/e)^{m^*} \sqrt{2 \pi m^*} e^{o(1)}} \frac{1}{2^{m^*}} 
 = \Theta(1)  \left(\frac{2 m^*}{e}\right)^{m^*} \nonumber \\
& = \Theta(1) \left(\frac{t k-\sqrt{t k/b p}+o(\sqrt{k})}{e} \right)^{t k/2 - \sqrt{t k/b p}/2 +o(\sqrt{k})} \nonumber\\
& = \left(\frac{t k}{e}\right)^{t k/2} \left(1 + O\left(\frac{\ln k}{\sqrt{k}}\right)\right)^k. \label{Match}
\end{align}

Now, substituting~\eqref{pmstar},~\eqref{Coef} and~\eqref{Match} into the expression for $f$ (given in~\eqref{EdgesExpansion}), we obtain the following:
\begin{align}
f(m^*)
& = b^{-\binom{k}{2}} (b p)^{t k/2}
\frac{1}{t!^k}
\left(\frac{t k}{e}\right)^{t k/2}
\left(1 + O\left(\frac{\ln k}{\sqrt{k}}\right)\right)^k \nonumber\\
& = \left( b^{-k+1} \left(\frac{t b p k}{e}\right)^t \frac{1}{t!^2} \right)^{k/2} \left(1 + O\left(\frac{\ln k}{\sqrt{k}}\right)\right)^k.
\label{fmstar}
\end{align}

\subsection*{Upper bound on $\ex \left(\alpha_t^{(k)} (G_{n,p}) \right)$}

By~\eqref{expansion} and~\eqref{fmstar}, we deduce that 
\begin{align}
\prob (A\in S_t)
\le \left(\frac{t k}{2} + 1\right) f(m^*)
 = \left( b^{-k+1} \left(\frac{t b p k}{e}\right)^t \frac{1}{t!^2} \right)^{k/2} \left(1 + O\left(\frac{\ln k}{\sqrt{k}}\right)\right)^k. \label{ProbtStable}
\end{align}
Thus, we obtain, 
\begin{align}
\ex (\sk)
& \le \binom{n}{k} \left( b^{-k+1} \left(\frac{t b p k}{e}\right)^t \frac{1}{t!^2} \right)^{k/2} \left(1 + O\left(\frac{\ln k}{\sqrt{k}}\right)\right)^k \nonumber \\
& = \left( 
e^2 n^2 b^{-k+1} k^{t-2} \left(\frac{t b p}{e}\right)^t \frac{1}{t!^2} \right)^{k/2} \left(1 + O\left(\frac{\ln k}{\sqrt{k}}\right)\right)^k. \label{FinalCalc}
\end{align}

Now, if we set 
\( k=\lceil\atp{t}+\eps(n) \rceil \)
for some function $\eps(n) \gg \ln \ln n/\sqrt{\ln n}$, then, 
substituting this into~\eqref{FinalCalc}, we obtain
\begin{align*}
\ex (\sk) \le \left( \left(1+O\left(\frac{\ln \ln n}{\ln n}\right)\right) b^{-\eps} \right)^{k/2} \left(1 + O\left(\frac{\ln k}{\sqrt{k}}\right)\right)^k = o(1), 
\end{align*}
thus proving the right-hand side inequality in Theorem~\ref{1stability}.

\subsection*{Lower bound on $\ex (\alpha_t^{(k)} (G_{n,p}) )$}

To derive the lower bound on $\ex (\sk)$, we observe 
\[ \ex (\sk) \ge  \binom{n}{k}\prob(A \in S_t, \; e(A)=m^*). \]
Let $(d_1,\ldots ,d_k)$ be a degree sequence such that, for every 
$1\le i\le k$, $d_i\le t$ and $\sum_i d_i
=2 m^*$.
By Theorem 2.16 in~\cite{Bol01}, with $\la:=\frac{1}{m^*}\sum_i \binom{d_i}{2}$, the number of graphs with
this degree sequence is 
\[ (1+o(1)) e^{-\la/2-\la^2/4} \frac{(2 m^*)!}{m^*! 2^{m^*}}. \]
But, since $d_i\le t$
for every $i$, then using the estimate from Lemma~\ref{mstar}
we obtain $\la \le t^2 k/2 m^* \le 2 t$ for $k$ large enough. So the total number of
graphs on $k$ vertices, $m^*$ edges and with maximum degree at most $t$ is
at least
\[ \frac{e^{-t - t^2}}{2} C_{2 m^*}(t,k)\frac{(2 m^*)!}{m^*!2^{m^*}}. \]
Since $k=O(\ln n)$, we have $\binom{n}{k}=\Omega(\sqrt{1/k})(n e/k)^k$.  Hence, using~\eqref{fmstar}, we obtain   
\begin{align}
\ex (\sk)
& \ge \binom{n}{k} \prob(A \in S_t, \; e(A)=m^*)
 \ge \binom{n}{k} \frac{e^{-t - t^2}}{2} f(m^*) \nonumber\\
& = \left( 
e^2 n^2 b^{-k+1} k^{t-2} \left(\frac{t b p}{e}\right)^t \frac{1}{t!^2} \right)^{k/2} \left(1 + O\left(\frac{\ln k}{\sqrt{k}}\right)\right)^k. \label{FinalCalcLow}
\end{align} 

If $k = \lfloor \atp{t}-\eps(n) \rfloor$  (\(> \atp{t}-\eps(n)-1\)) where $\eps(n)$ is some function satisfying $\ln \ln n/\sqrt{\ln n} \ll \eps(n) \ll \ln n$, then by~\eqref{FinalCalcLow} we obtain
\begin{align}
\ex (\sk)
& \ge \left( \left(1+O\left(\frac{\ln \ln n}{\ln n}\right)\right) b^{ \eps (n)}\right)^{k/2} \left(1 + O\left(\frac{\ln k}{\sqrt{k}}\right)\right)^k
 = n^{\eps(n)(1 + o(1)) } \rightarrow \infty. \label{explow}
\end{align}
In the next section, we use a sharp concentration inequality to show moreover that the following holds. 
\begin{lemma} \label{ExpConc} If $\eps(n) \gg \ln \ln n/\sqrt{\ln n}$ is a function that satisfies $\limsup_{n \to \infty} \eps(n) \le 2$, then  
\[
\prob \left( \alpha_t (G_{n,p}) < \lfloor \atp{t}-\eps(n) \rfloor \right) \le \exp \left( - n^{\eps(n)(1+o(1))} \right).
\]
\end{lemma}
Since the right-hand side is $o(1)$, we obtain the lower bound of Theorem~\ref{1stability}.  This lemma will also be a key step in the proof of the upper bound of Theorem~\ref{tchrom}, when we need the fact that the right-hand side tends to $0$ quickly.


\section{A second moment calculation - Proof of Lemma~\ref{ExpConc}} \label{2ndMom}
Let  $(x_n)$ be a bounded sequence of real numbers such that for
\[ k = 2\log_b n+(t-2)\log_b\log_b n+x_n \in \mathbb{N} \]
we have $\ex (\sk)\rightarrow \infty$ as $n \rightarrow \infty$.  
In this section, we prove that a.a.s.~there is a $k$-subset of $V_n$ which is $t$-stable, using a second moment argument.  For this, we use Janson's Inequality (\cite{Jan90},~\cite{JLR90} or Theorems~2.14,~2.18 in~\cite{JLR00}):
\begin{align} \label{Janson} \prob (\sk = 0) \le \exp \left( - \frac{\ex^2(\sk)}{\ex(\sk) + \Delta }\right),
\end{align}
where   
\[
\Delta = \sum_{A, B \subseteq V_n, k-1\ge |A \cap B|\ge 2} \prob (A, B \in S_t). 
\]
Let $p(k,\ell)$ be the probability that two $k$-subsets of $V_n$ that
overlap on exactly $\ell$ vertices are both in $S_t$.
We write
\begin{align*} 
\Delta = & \sum_{\ell=2}^{k-\lfloor (t+3) \log_b \log_b n\rfloor} \binom{n}{k} \binom{k}{\ell} \binom{n - k}{k-\ell} p(k,\ell) \nonumber \\
  & + \sum_{\ell=k-\lfloor (t+3) \log_b \log_b n \rfloor+1}^{k-1} \binom{n}{k} \binom{k}{\ell} \binom{n - k}{k-\ell} p(k,\ell) 
=:  \Delta_1 + \Delta_2.
\end{align*}
We conclude the proof of Lemma~\ref{ExpConc} by showing that 
\[\Delta_1 = O\left(\frac{\ln^5 n}{n^2}\right) \ex^2 (\sk) \ \mbox{and} \ \Delta_2 = o(\ex (\sk)).\]
By~\eqref{explow}, $\ex (\sk) \ge n^{\eps(n)(1 + o(1))}$.
If $\ex (\sk) < \Delta$, then by the above estimates for $\Delta_1$ and $\Delta_2$, we have $\ex(\sk)+\Delta = O(\ln^5 n/n^2) \ex^2(\sk)$ and, by~\eqref{Janson},
\[ \prob (\sk = 0) \le \exp \left( -\Omega \left(\frac{n^2}{\ln^5 n}\right)\right) = \exp\left(-n^{2 + o(1)}\right) \le \exp\left(-n^{\eps(n)(1 + o(1))}\right)\]
(where the last inequality uses $\limsup_{n\to\infty}\eps(n)\le 2$). Otherwise, we have $\ex(\sk)+\Delta\le2\ex(\sk)$ and
\[ \prob (\sk = 0) \le \exp \left( -\frac12 \ex(\sk)\right) \le \exp\left(-n^{\eps(n)(1 + o(1))}\right).\]

\subsubsection*{{\bf Bounding} $\Delta_1$}

Let us begin by bounding $\Delta_1$, first estimating $p(k,\ell)$. 
Let $A$ and $B$ be two $k$-subsets of $V_n$ that
overlap on exactly $\ell$ vertices, i.e.~$|A \cap B|=\ell$. Then $p(k,\ell)=\prob (A,B \in S_t) = \prob (A \in S_t \; | \; B \in S_t)
\prob (B \in S_t)$.

The property of having maximum degree at most $t$ is monotone
decreasing; so if we condition on the
set $E$ of edges induced by $A\cap B$, then the conditional
probability that $A \in \mathcal{S}_{t}$ is maximized when $E =
\emptyset$.  Thus,
\begin{equation*}
\prob \left( A \in S_{t} \ | \ B \in S_{t} \right)
\le \prob \left( A \in S_{t} \ | \  E = \emptyset \right)
\le \frac{\prob (A \in S_{t})}{\prob(E = \emptyset)} =
b^{\binom{\ell}{2}} \prob (A \in S_{t}).
\end{equation*}
Therefore,
\begin{align}
p(k,\ell)
 = \prob (A \in S_t \; | \; B \in S_t)\prob(B \in S_t)
 \le b^{\binom{\ell}{2}} \left(\prob(A \in S_t)\right)^2. \label{Bound}
\end{align}
On the other hand, for every $\ell\le k$,
\[ \binom{k}{\ell} \binom{n-k}{k-\ell} \le k^\ell \frac{k^\ell}{(n-k)^{\ell}}\binom{n}{k}. \]
Using the estimate of~\eqref{Bound} along with the above inequality, we have 
\begin{align} \label{Delta1} 
\Delta_1 & \le  \left(\binom{n}{k}\prob(A \in S_t) \right)^2 \ 
\sum_{\ell=2}^{k-\lfloor (t+3) \log_b \log_b n \rfloor} 
\left(\frac{k^2}{n-k}\right)^\ell b^{\binom{\ell}{2}}\nonumber \\
& \le \ex^2 (\alpha_t^{(k)} (G_{n,p})) \sum_{\ell=2}^{k-\lfloor (t+3) \log_b \log_b n \rfloor} \left(\frac{k^2}{n-k}\right)^\ell b^{\binom{\ell}{2}}.
\end{align}
If we set $s_{\ell} = (k^2/(n-k))^\ell  b^{\binom{\ell}{2}}$, then 
\(
s_{\ell + 1}/s_{\ell} = b^{\ell} k^2/(n-k).
\) 
So the sequence $\{s_{\ell}\}$ is 
strictly decreasing for $\ell < \log_b (n-k) - 2 \log_b k$ and is strictly increasing for 
$\ell > \log_b (n-k) - 2 \log_b k$. So 
\[ \max \{ s_{\ell} : 2 \le \ell \le k- \lfloor(t+3) \log_b \log_b n \rfloor \} \le 
\max \left\{ s_2 ,  s_{\lceil 2\log_b n-4.5\log_b \log_b n \rceil} \right\}. \]
We have that $s_2 = b k^4/(n-k)^2$, but
\begin{align*}
s_{\lceil 2\log_b n-4.5\log_b \log_b n \rceil}
&\le \left( \frac{k^2}{n-k} b^{\log_b n-2.25\log_b \log_b n}\right)^{2\log_b n-4.5\log_b \log_b n} \\
&\le \left( \frac{4 \log_b^2 n}{\log_b^{2.25} n} \right)^{2\log_b n-4.5\log_b \log_b n}
\le \left( \frac{4}{\log_b^{0.25} n} \right)^{\log_b n} = o(s_2).
\end{align*}
Thus, Inequality~\eqref{Delta1} now becomes for $n$ large enough
\begin{align*} 
\Delta_1 \le \frac{b k^5}{(n-k)^2} \ex^2 (\alpha_t^{(k)}(G_{n,p})) = O\left(\frac{\ln^5 n}{n^2}\right)\ex^2 (\alpha_t^{(k)}(G_{n,p})). 
\end{align*}

\subsubsection*{{\bf Bounding } $\Delta_2$} 
 
Now, we will show that 
\(\Delta_2 = o(\ex(\alpha_t^{(k)} (G_{n,p})))\).
First, we have 
\begin{align*}
\binom{k}{\ell} \binom{n-k}{k-\ell} \le (k n)^{k-\ell}.
\end{align*}
We now give a rough estimate on $p(k,\ell)$.
If $A, \ B$ are two $k$-sets of vertices that overlap on $\ell$ vertices (and if $\deg_S(v)$ denotes the number of neighbours of $v$ in $S$),
then 
\begin{align*}
\prob (B \in S_t \; | \; A \in S_t)
& \le \prob (\forall v \in B\setminus A, \ \deg_{A\cap B}(v) \le t)
 \le \left( \binom{\ell}{\ell-t}(1-p)^{\ell-t} \right)^{k-\ell} \\
& \le \left( k^{t} b^{t-\ell}\right)^{k-\ell}
 \le b^{\left(t\log_b k+t-k+\lfloor (t+3) \log_b \log_b n \rfloor \right)(k-\ell)}\\
& = b^{\left(-2 \log_b n + (t + 5) \log_b \log_b n + \Theta(1)\right)(k-\ell)}
 \le \left(\frac{\log_b^{t+6} n}{n^2}\right)^{k-\ell}.
\end{align*}
Substituting these estimates into the expression for $\Delta_2$, we obtain 
\begin{align*}
\Delta_2 & \le \binom{n}{k} \prob (A\in S_t)\sum_{\ell=k-\lfloor (t+3) \log_b \log_b n \rfloor+1}^{k-1}
\left(k n\frac{\log_b^{t+6} n}{n^2 } \right)^{k-\ell} \\
& \le \ex (\alpha_t^{(k)} (G_{n,p})) k \left(\frac{k \log_b^{t+6} n}{n} \right) 
 = o(\ex (\alpha_t^{(k)} (G_{n,p}))).
\end{align*}

\section{The $t$-improper chromatic number} \label{chrom}

\subsection{The upper bound}
Our general approach follows Bollob\'as~\cite{Bol88} --- see also \cite{McD90}.  We revisit the analysis in order to obtain an improved upper bound to match the lower bound of Panagiotou and Steger~\cite{PaSt09}.
For a fixed  $0 < \eps < 1$,
we set $\hat{\alpha}_{t,p} (n) = \lfloor \atp{t} -1 - \eps \rfloor $. 
First, we will show the following.
\begin{lemma} 
A.a.s.~for all $V' \subseteq V_n$ with $|V'| \ge n/\ln^3 n$, 
we have $\alpha_t (G_{n,p}[V']) \ge \hat{\alpha}_{t,p}(|V'|)$. 
\end{lemma}
\begin{proof}
Note that~\eqref{explow} implies that 
for any $V'\subseteq V_n$ with $|V'| \ge n/\ln^3 n$, we have 
\[ \ex \left( \alpha_{t,p}^{(\hat{\alpha}_{t,p}(|V'|))} (G_{n,p}[V']) \right) \ge |V'|^{1+\eps +o(1)}.\] 
So, applying Lemma~\ref{ExpConc}, we deduce that 
\[ \prob \left( \alpha_t (G_{n,p}[V']) < \hat{\alpha}_{t,p}(|V'|) \right) = \exp \left( - |V'|^{1+\eps+o(1)}\right) \le 
\exp \left(-\left(\frac{n}{\ln^3 n}\right)^{1+\eps+o(1)} \right).\]
Since there are at most $2^n$ choices for $V'$, the probability that there exists a set $V' \subseteq V_n$ 
with $|V'| \ge n/\ln^3 n$ and $\alpha_t (G_{n,p}[V']) < \hat{\alpha}_{t,p}(|V'|)$ is at most 
$2^n \exp \left( -(n/\ln^3 n)^{1+\eps+o(1)}\right) = o(1)$. 
\end{proof}

We consider the following algorithm for $t$-improperly colouring $G_{n,p}$. Let $V'=V_n$.  While $|V'|\ge n/\ln^3 n$, 
we choose and remove a $t$-stable set from $G_{n,p}[V']$ of size $\hat{\alpha}_t (|V'|)$.  
At the end, we obtain a collection of $t$-stable sets and each of them will form a colour class.
The above lemma implies that 
a.a.s.~we will be able to perform this algorithm, and end up with a set of at most $n/\ln^3 n$ vertices. We give a 
different a colour to each of these vertices. Thus, if the above algorithm ``runs" for $f(n)$ steps,  then
$\chi_t(G_{n,p}) \le f(n) + n/\ln^3 n$.

Since $\alpha_{t,p}(s)- 1 - \eps$ is strictly increasing for all $s$ that 
are sufficiently large, for these $s$ the function $\hat{\alpha}_{t,p} (s)$ is non-decreasing. 
It is easy to see that 
\begin{align*}
\hat{\alpha}_{t,p}\left(\left\lceil\frac{n}{\ln^3 n}\right\rceil\right)
& = 2\log_b n \left( 1+O\left(\frac{\ln \ln n}{\ln n} \right)\right)
 = \hat{\alpha}_{t,p}(n) \left( 1+O\left(\frac{\ln \ln n}{\ln n} \right)\right).
\end{align*}
Since $\hat{\alpha}_{t,p} (\lceil n/\ln^3 n \rceil)\le \hat{\alpha}_{t,p}(s) \le \hat{\alpha}_{t,p}(n)$ 
for all integers  $n/\ln^3 n \le s \le n$, 
\begin{align} \label{alphaest}
\hat{\alpha}_{t,p} (s) = \hat{\alpha}_{t,p}(n) \left( 1+O\left(\frac{\ln \ln n}{\ln n} \right)\right),
\end{align}
and therefore
\begin{align} \label{fnrough}
f(n)=\frac{n}{\hat{\alpha}_{t,p}(n)}\left(1+O\left(\frac{\ln \ln n}{\ln n}\right)\right).
\end{align}
Assume that there are $n_i$ vertices available
when we have removed $i$ $t$-stable sets from $V_n$.
Thus, the $t$-stable set that will be picked during the $(i+1)$th iteration will have size $\hat{\alpha}_{t,p} (n_i)$. 
Since the colouring algorithm stops as soon as there are less than $n/\ln^3 n$ vertices available, 
the following inequality holds:
\begin{align} \label{fnbound}
 \sum_{i=0}^{f(n)-2}  \hat{\alpha}_{t,p}(n_i) \le n\left(1 - \frac{1}{\ln^3 n}\right) \le n. 
\end{align}
Note that for all $i\ge 0$, $n_i = n -\sum_{j=0}^{i-1} \hat{\alpha}_{t,p} (n_j)$.
Therefore, 
\begin{align*}
\log_b n_i & = \log_b \left(n- \sum_{j=0}^{i-1} \hat{\alpha}_{t,p} (n_j)\right) 
 = \log_b n + \log_b \left(1 - \frac{\sum_{j=0}^{i-1} \hat{\alpha}_{t,p} (n_j)}{n} \right).
\end{align*}
We have\COMMENT{Note that $\int \ln (1-x) dx = -(1-x)\ln (1-x) + 1-x$.}
\begin{align*} 
\sum_{i=0}^{f(n)-2} \log_b \left(1 - \frac{\sum_{j=0}^{i-1}\hat{\alpha}_{t,p} (n_j)}{n} \right)
& =\frac{1}{\ln b}\sum_{i=0}^{f(n)-2} \frac{n}{\hat{\alpha}_{t,p}(n_i)} \ln \left(1 - \frac{\sum_{j=0}^{i-1} \hat{\alpha}_{t,p} (n_j)}{n} \right)  \frac{\hat{\alpha}_{t,p}(n_i)}{n}  \\
& \stackrel{\eqref{alphaest}}{=} \frac{n (1+o(1))}{\hat{\alpha}_{t,p}(n) \ln b} \sum_{i=0}^{f(n)-2} \ln \left(1 - \frac{\sum_{j=0}^{i-1} \hat{\alpha}_{t,p} (n_j)}{n} \right)  \frac{\hat{\alpha}_{t,p}(n_i)}{n} \\
& =\frac{n (1+o(1))}{\hat{\alpha}_{t,p}(n) \ln b} \int_{0}^{1} \ln (1-x)d x
 = -\frac{n (1+o(1))}{\hat{\alpha}_{t,p}(n) \ln b}. 
\end{align*}
So 
\begin{align} \label{logs}
\sum_{i=0}^{f(n)-2} 2\log_b n_i = (f(n)-1)2\log_b n 
- \frac{2n (1+o(1))}{\hat{\alpha}_{t,p}(n) \ln b}.
\end{align}

Also,
\begin{align*}
\log_b \log_b n_i 
& \ge \log_b \log_b \left(\frac{n}{\ln^3 n}\right)
= \log_b \log_b n + \log_b \left(1 - \frac{3\log_b \ln n}{\log_b n} \right) \\
& = \log_b \log_b n - O\left(\frac{\ln \ln n}{\ln n} \right).
\end{align*}
Moreover, $\log_b \log_b n_i \le \log_b \log_b n$ so, for every $t\ge 0$, 
\begin{align} \label{loglogn}
(t-2) \sum_{i=0}^{f(n)-2} \log_b \log_b n_i \ge
(f(n)-1)(t-2)\log_b \log_b n - O\left(\frac{f(n) \ln \ln n}{\ln n} \right).
\end{align}
Now, Equality~\eqref{logs} and Inequality~\eqref{loglogn} imply that 
for every $t \ge 0$ we have 
\begin{align*}
\sum_{i=0}^{f(n)-2}  \hat{\alpha}_{t,p} (n_i) 
& \ge (f(n)-1)\left( \atp{t} - \eps - 2 \right) 
-\frac{2n (1+o(1))}{\hat{\alpha}_{t,p}(n) \ln b}
- O\left(\frac{f(n) \ln \ln n}{\ln n} \right) \\
& \ge (f(n)-1)\left( \atp{t} -\eps - 2  - \frac{2n (1+o(1))}{f(n) \hat{\alpha}_{t,p}(n) \ln b} 
- O\left(\frac{\ln \ln n}{\ln n} \right) \right) \\
& \stackrel{\eqref{fnrough}}{=} (f(n)-1) \left( \atp{t} -\eps - 2
- \frac{2}{\ln b}- o(1)\right). 
\end{align*}
So by~\eqref{fnbound} we obtain
\begin{align*} 
f(n)-1\le \frac{n}{\atp{t} - 2/\ln b -2 - \eps - o(1)}. 
\end{align*}

\subsection{The lower bound}
This proof is the generalisation of a proof of the lower bound on the chromatic number of a dense 
random graph given recently by Panagiotou and Steger~\cite{PaSt09}.
We let $\alpha_C (n) = 2\log_b n + (t-2)\log_b \log_b n - C$, where $C =C_n > 2\log_b n + (t-2)\log_b \log_b n - \atp{t}$
is some function which is $\Theta (1)$, such that $\alpha_C(n)$ is integral. We specify $C$ at a later stage. 
Let $r= r_C := \lfloor n/\alpha_C (n) \rfloor$.
By Theorem~\ref{1stability}, a.a.s.~there are no $t$-stable sets in $G_{n,p}$ of size more than $\atp{t}+1$. 
(In fact, according to Theorem~\ref{1stability}, we could have used the bound $\atp{t}+\eps$, but this would not give any improvement.) 
We will estimate the expected number of $t$-improper colourings of $G_{n,p}$ with 
$r$ colours such that each colour set has size at most $\atp{t}+1$.
In particular, we show that, if
$C < 2\log_b n + (t-2)\log_b \log_b n - \atp{t} + 1 + 2/\ln b - \eps$,
then this expectation converges to zero, proving that $\chi_t (G_{n,p}) > r_C$ a.a.s. 

Let $\D$ denote the set of $r$-tuples of positive integers $(k_1,\ldots , k_r)$ such that $\sum_{i=1}^r k_i =n$ and $k_i \le \atp{t} + 1$ for all $i$. For some $(k_1,\ldots , k_r) \in \D$, let $\Pa = (P_1,\ldots, P_r)$ denote a partition of $V_n$ into 
$r$ non-empty parts $P_1,\ldots, P_r$ such that $|P_i|=k_i$. 
From~\eqref{ProbtStable}, we obtain
\begin{align*}
\prob (P_i\in S_t)
\le \left( b^{-k_i+1} \left(\frac{t b p k_i}{e}\right)^t \frac{1}{t!^2} \right)^{k_i/2} \left(1 + O\left(\frac{\ln k_i}{\sqrt{k_i}}\right)\right)^{k_i}.
\end{align*}
 
\begin{align*}
\prob (P_i \in S_t, \ \forall i)
& = \prod_{i=1}^r \prob (P_i \in S_t)
 \le \prod_{i=1}^r \left( b^{-k_i+1} \left(\frac{t b p k_i}{e}\right)^t \frac{1}{t!^2} \right)^{k_i/2} \left(1 + O\left(\frac{\ln k_i}{\sqrt{k_i}}\right)\right)^{k_i} \nonumber \\
&= b^{-\left(\sum_{i=1}^r {k_i}^2/2\right) + n/2} 
\left(\frac{t b p}{e}\right)^{t n/2}\left(\prod_{i=1}^r {k_i}^{t k_i/2} \right)
\frac{1}{t!^{n}}
(1 + o(1))^n \nonumber \\
&= \left(\frac{t b^{1 + 1/t}p}{e t!^{2/t}}\right)^{t n/2} b^{-\sum_{i=1}^r {k_i}^2/2} 
\left(\prod_{i=1}^r {k_i}^{t k_i/2}\right)
(1 + o(1))^n,
\end{align*} 
uniformly over all $(k_1,\ldots, k_r) \in \D$.
So, if $X_{t,r} = X_{t,r}(G_{n,p})$ denotes the number of $t$-improper colourings with $r$ colours and with each colour class of size at most 
$\atp{t} +1$, then 
\begin{align} \label{ExpectCols}
\ex (X_{t,r}) = \frac{1}{r!} \left(\frac{t b^{1 + 1/t}p}{e t!^{2/t}}\right)^{t n/2} 
\sum_{(k_1,\ldots, k_r) \in \D} \binom{n}{k_1 \cdots k_r} b^{-\sum_{i=1}^r\left(\frac{k_i^2}{2} - \frac{t}{2} k_i \log_b k_i\right)} (1 + o(1))^n.
\end{align}
We call a partition where all parts differ by at most one pairwise \emph{balanced}. 
In the next subsection, we give a routine proof of the following property of balanced partitions.
\begin{lemma} \label{MaxBal}
For large enough $n$, the function
\[h(P) :=-\sum_{i=1}^r \left(\frac{k_i^2}{2} - \frac{t}{2} k_i \log_b k_i\right),\]
where 
$P=\{P_1,\ldots, P_r \}$ is a partition of $V_n$ with $|P_i|=k_i$, is maximised over $\D$ when $P$ is a balanced partition.
\end{lemma}

Let $B$ be a balanced partition. Then all parts have sizes either equal to $\alpha_C (n)$ or to $\alpha_C(n)+1$ and there are less than
$\alpha_C(n)$ parts that take the latter quantity. 
Then 
\begin{align} \label{Balanced}
h(B) &= -\frac{n}{\alpha_C (n)} \left( \frac{\alpha_C^2(n)}{2} - \frac{t}{2} \alpha_C(n) \log_b \alpha_C(n) \right) +o(n) \nonumber \\
&= -\frac{1}{2} n \alpha_C(n) + \frac{t}{2}n \log_b \alpha_C(n) +o(n) \nonumber \\
&= -n \log_b n - \frac{t-2}{2}n \log_b \log_b n +\frac{Cn}{2} + \frac{t}{2}n \log_b 2 + \frac{t}{2}n \log_b \log_b n  +o(n) \nonumber \\
&= -n \log_b n + n \log_b \log_b n + \frac{Cn}{2} + \frac{t}{2}n \log_b 2 + o(n).
\end{align}
Also, for any $(k_1,\ldots, k_r) \in \D$, we have (for $n$ large enough) 
\begin{align} 
\binom{n}{k_1 \cdots k_r} &\le \frac{n!}{\left(\alpha_C(n)!\right)^{r}} = O(n^{1/2}) \frac{n^n}{(\alpha_C(n))^{n} 
\left(\sqrt{2\pi \alpha_C(n)}\right)^{r}} \le \frac{n^n}{(\alpha_C(n))^{n+r/2} } \nonumber \\
&= b^{n \log_b n - n \log_b \alpha_C(n) - \frac{r}{2} \log_b \alpha_C(n)} = b^{n \log_b n - n \log_b 2 - n \log_b \log_b n +o(n)} \label{Multi}
\end{align}
since $r \log_b \alpha_C (n) \le (n/\alpha_C(n)) \log_b \alpha_C(n) =o(n)$.
Finally, $r!\ge r^re^{-r}$ and therefore
\begin{align} \label{factorial}
\frac{1}{r!} \le b^{-r\log_b r +r\log_b e} = b^{-\frac{n}{\alpha_C(n)} \log_b \left(\frac{n}{\alpha_C(n)} \right) +o(n)} 
= b^{-n \frac{\log_b n}{\alpha_C(n)}  +o(n)} = b^{-\frac{n}{2} + o(n)}.
\end{align}
As there are at most $\binom{n}{r}\le (e n/r)^{r} \le 
\left({2 e \alpha_C(n)} \right)^{r}
\le 
b^{r\log_b \alpha_C(n) + O(r)} = b^{o(n)}$ summands in~\eqref{ExpectCols}, we obtain from~\eqref{Balanced},~\eqref{Multi} and~\eqref{factorial} that
\begin{align*}
\ex (X_{t,r})&\le \left(\frac{t b^{1 + 1/t}p}{e t!^{2/t}}\right)^{t n/2} 
b^{\frac{Cn}{2} + \frac{t}{2}n \log_b 2- n \log_b 2 -\frac{n}{2}} (1 + o(1))^n\\ 
&=  b^{ \frac{n}{2}\left(t\log_b (2 t p/e) + t  -2 \log_b t!+ C - 2 \log_b 2 \right) } (1 + o(1))^n.
\end{align*}
Therefore, if $C = C_n < -\log_b (t^t/t!^2) - t\log_b (2 b p/e) - \log_b (1/4) - \eps$, i.e.~if $\alpha_C(n) > \atp{t} - 2/\ln b -1 + \eps$ for an arbitrary $\eps>0$, then $\ex (X_{t,r}) =o(1)$.
Thus, a.a.s.
\begin{align*} \chi_t (G_{n,p})
& \ge \frac{n}{\atp{t} -\frac{2}{\ln b} -1 + \eps }. 
\end{align*}

\subsection{Proof of Lemma~\ref{MaxBal}}
Suppose $h(P)$, $P$, $k_i$ are defined as in Lemma~\ref{MaxBal} and furthermore assume that the parts of $P$ are ordered by increasing size, 
i.e.,~$k_1 \le \cdots \le k_r$.
Let $\tilde{P}=\{\tilde{P}_1,\ldots, \tilde{P}_r \}$ be a partition of $V_n$ where for some 
$v \in P_r$ we have $\tilde{P}_1=P_1 \cup \{ v\}$ and $\tilde{P}_r=P_r\setminus \{ v\}$, whereas $\tilde{P}_i=P_i$ for all $1 < i < r$. In other words, we obtain $\tilde{P}$ by moving a vertex from $P_r$ to $P_1$. Lemma~\ref{MaxBal} easily follows from the repeated application of the following.
\begin{lemma} \label{Swap}
For large enough $n$, it holds that, if $k_1 < k_r - 1$, then $h(\tilde{P}) > h(P)$. 
\end{lemma}

\begin{proof}
First, $k_1 \le \alpha_C(n)$ and $k_r \ge \alpha_C(n) + 1$, since the number of parts is $r = \lfloor n /\alpha_C(n) \rfloor$.
\begin{align} \label{Difference} 
2(h(\tilde{P})-h(P)) = &-(k_1+1)^2 + t(k_1+1) \log_b (k_1 + 1) 
-(k_r-1)^2 + t(k_r-1) \log_b (k_r - 1) \nonumber \\
& + {k_1}^2 -t k_1 \log_b k_1 + {k_r}^2 - t k_r\log_b k_r \nonumber \\
= & 2(k_r-k_1-1) + t ((k_1+1) \log_b (k_1 + 1) - k_1 \log_b k_1) \nonumber \\
& + t((k_r-1) \log_b (k_r - 1) - k_r\log_b k_r). 
\end{align}
Note that 
\begin{align*}
(k_1+1) \log_b (k_1 + 1) & = (k_1+1) \log_b k_1 + (k_1+1) \log_b \left( 1+1/k_1 \right) \\
& \ge (k_1+1) \log_b k_1 + (k_1+1)\left( 1/k_1 - 1/(2 {k_1}^2) \right) \\
& = k_1 \log_b k_1 + \log_b k_1 + 1 + O\left(1/k_1 \right),
\end{align*}
and similarly, since $k_r\ge \alpha_C (n) + 1 \to \infty$ as $n \to \infty$, 
\begin{align*}
(k_r-1) \log_b (k_r - 1)
& = k_r \log_b k_r - \log_b k_r + 1 - o(1).
\end{align*}
Substituting these estimates into~\eqref{Difference}, we obtain
\begin{align} \label{Diff1} 
2(h(\tilde{P})-h(P)) \ge 2(k_r-k_1-1) - t(\log_b k_r - \log_b k_1) + O\left( 1/k_1 \right).
\end{align}
Assume first that $k_r - k_1 \le \ln \ln n$. 
Then 
\(\log_b (k_r/k_1) \le \log_b (k_r/(k_r -\ln \ln n)) = \log_b ( 1 + \ln \ln n/(k_r -\ln \ln n) ) =o(1)\).
But $k_r -k_1-1\ge 1$ and $k_1 \ge \alpha_C(n) + 1 - \ln \ln n$ and therefore the right-hand side of~\eqref{Diff1} is positive for $n$ large enough. 
If, on the other hand $k_r - k_1 > \ln \ln n$, we write $\log_b k_r = \log_b (k_r - k_1 + k_1) = \log_b (k_r-k_1) + \log_b \left(1+k_1/(k_r- k_1) \right)$. 
So 
\begin{align*}
\log_b k_r - \log_b k_1 & = \log_b (k_r-k_1) + \log_b \left( 1+k_1/(k_r- k_1) \right) - \log_b k_1 \\
& = \log_b (k_r-k_1) + \log_b \left( 1/k_1+1/(k_r- k_1) \right) \\
& = \log_b (k_r-k_1) + \log_b \left( 1/k_1+o(1) \right)
 \le \log_b (k_r-k_1) + 1.
\end{align*}
So 
\[2(k_r-k_1-1) - t (\log_b k_r - \log_b k_1) + O\left( 1/k_1 \right)\ge 2(k_r -k_1 -1) - t (\log_b (k_r-k_1) - 1) \rightarrow \infty\]
as $n \to \infty$ and, by~\eqref{Diff1}, $h(\tilde{P})-h(P) >0$ for $n$ large enough.
\end{proof}

\subsection*{Acknowledgement}
We would like to thank an anonymous referee for a number of helpful comments.

\bibliographystyle{abbrv}
\bibliography{tstable}

\end{document}